\documentclass[12pt, reqno]{amsart}%
\usepackage{amsmath, amsthm, amscd, amsfonts, amssymb, graphicx, color}
\usepackage[bookmarksnumbered, colorlinks, plainpages]{hyperref}
\usepackage{amsmath}
\usepackage{amsfonts}
\usepackage{amssymb}
\usepackage{graphicx}%
\providecommand{\U}[1]{\protect \rule{.1in}{.1in}}
\newtheorem{theorem}{Theorem}[section]

\theoremstyle{definition}
\newtheorem{definition}[theorem]{Definition}

\theoremstyle{remark}

\numberwithin{equation}{section}
\begin{document}
\title[Quaternionic Mannheim curves in Euclidean space $E^{4}$]{Characterizations of the Quaternionic Mannheim curves in Euclidean space
$E^{4}$}
\author{O. Zek\.{I} Okuyucu}
\address{Department of Mathematics, University of Bilecik
\c{S}eyh Edebali, Turkey }
\email{osman.okuyucu@bilecik.edu.tr}

\begin{abstract}
In \cite{matsuda}, \textit{Matsuda and Yorozo} obtained that Mannheim curves
in 4-dimensional Euclidean space. In this study, we define a
\textbf{quaternionic Mannheim curve} and we give some characterizations of
them in Euclidean $3-$space and $4-$space.

\end{abstract}
\maketitle

\section{Introduction}

The geometry of curves in a Euclidean space have been developed a long time
ago and we have a deep knowledge about it. In the theory of curves in
Euclidean space, one of the important and interesting problem is
characterizations of a regular curve. We can characterize some curves via
their relations between the Frenet vectors of them. For instance Mannheim
curve is a special curve and it is characterized using the Frenet vectors of
its Mannheim curve couple.

In 2007, the definition of Mannheim curves in Euclidean 3-space is given by H.
Liu and F. Wang \cite{liu} with the following:

\begin{definition}
Let $\alpha$ and $\beta$ be two curves in Euclidean 3-space If there exists a
corresponding relationship between the space curves $\alpha$ and $\beta$ such
that, at the corresponding points of the curves, the principal normal lines of
$\alpha$ coincides with the binormal lines of $\beta$, then $\alpha$ is called
a Mannheim curve, and $\beta$ is called a Mannheim partner curve of $\alpha$.
\end{definition}

In their paper, they proved that a given curve is a Mannheim curve if and only
if then for $\lambda \in \mathbb{R}$ it has $\lambda \left(  \varkappa^{2}%
+\tau^{2}\right)  =\varkappa$ curve Also in 2009, Matsuda and Yorozo, in
\cite{matsuda}, defined generalized Mannheim curves in Euclidean 4-space. If
the first Frenet vector at each point of $\alpha$ is included in the plane
generated by the second Frenet vector and the third Frenet vector of $\beta$
at corresponding point under a bijection, which is from $\alpha$ to $\beta$.
Then the curve $\alpha$ is called generalized Mannheim curve and the curve
$\beta$ is called generalized Mannheim mate curve of $\alpha$. And they gave a
theorem such that if the curve $\alpha$ is a generalized Mannheim curve in
Euclidean 4-space, then the first curvature function $k_{1}$ and second
curvature functions $k_{2}$ of the curve $\alpha$ satisfy the equality:%
\[
k_{1}(s)=\mu \left \{  \left(  k_{1}(s)\right)  ^{2}+\left(  k_{2}(s)\right)
^{2}\right \}
\]
where $\mu$ is a positive constant number.

The quaternion was introduced by Hamilton. His initial attempt to generalize
the complex numbers by introducing a three-dimensional object failed in the
sense that the algebra he constructed for these three-dimensional object did
not have the desired proporties. On the 16th October 1843 Hamilton discovered
that the appropriate generalization is one in which the scalar(real) axis is
left unchanged whereas the vector(imaginary) axis is supplemented by adding
two further vectors axes.

In 1987, The Serret-Frenet formulas for a quaternionic curve in $E^{3}$ and
$E^{4}$ was defined by Bharathi and Nagaraj \cite{nag} and then in 2004,
Serret-Frenet formulas for quaternionic curves and quaternionic inclined
curves have been defined in Semi-Euclidean space by \c{C}\"{o}ken and Tuna in
2004 \cite{coken}.

In 2011 G\"{u}ng\"{o}r and Tosun studied quaternionic rectifying curves
\cite{gungor}. Also, G\"{o}k et.al \cite{gok1,ferda} defined a new kind of
slant helix in Euclidean space $E^{4}$ and semi-Euclidean space $E_{2}^{4}.$
It called quaternionic $B_{2}$-slant helix in Euclidean space $E^{4}$ and
semi-real quaternionic $B_{2}$-slant helix in semi-Euclidean space $E_{2}%
^{4},$ respectively$.$ Recently, Sa\u{g}lam, in \cite{saglam}, has studied on
the osculating spheres of a real quaternionic curve in Euclidean $4-$space.

In this study, we define quaternionic Mannheim curves\textbf{\ }and we give
some characterizations of them in Euclidean $3$ and $4$ space.

\section{Preliminaries}

Let $Q_{H}$ denotes a four dimensional vector space over the field $H$ of
characteristic grater than 2. Let $e_{i}$ $(1\leq i\leq4)$denote a basis for
the vector space. Let the rule of multiplication on $Q_{H}$ be defined on
$e_{i}$ $(1\leq i\leq4)$ and extended to the whole of the vector space by
distributivity as follows:

A real quaternion is defined with $q=a\overrightarrow{e_{1}}+b\overrightarrow
{e_{2}}+c\overrightarrow{e_{3}}+de_{4}$ where $a,b,c,d$ are ordinary numbers.
Such that%
\begin{equation}%
\begin{array}
[c]{cc}%
e_{4}=1, & e_{1}^{2}=e_{2}^{2}=e_{3}^{2}=-1,\\
e_{1}e_{2}=e_{3}, & e_{2}e_{3}=e_{1},\text{ }e_{3}e_{1}=e_{2},\\
e_{2}e_{1}=-e_{3}, & e_{3}e_{2}=-e_{1},\text{ }e_{1}e_{3}=-e_{2}.
\end{array}
\label{2-1}%
\end{equation}
If we denote $S_{q}=d$ and $\overrightarrow{V_{q}}=a\overrightarrow{e_{1}%
}+b\overrightarrow{e_{2}}+c\overrightarrow{e_{3}}$, we can rewrite real
quaternions the basic algebraic form $q=S_{q}+\overrightarrow{V_{q}}$ where
$S_{q}$ is scalar part of $q$ and $\overrightarrow{V_{q}}$ is vectorial part.
Using these basic products we can now expand the product of two quaternions to
give%
\begin{equation}
p\times q=S_{p}S_{q}-\langle \overrightarrow{V_{p}},\text{ }\overrightarrow
{V_{q}}\rangle+S_{p}\overrightarrow{V_{q}}+S_{q}\overrightarrow{V_{p}%
}+\overrightarrow{V_{p}}\wedge \overrightarrow{V_{q}}\text{ for every }p,q\in
Q_{H}, \label{2-2}%
\end{equation}
where we have use the inner and cross products in Euclidean space $E^{3}$
\cite{nag}$.$ There is a unique involutory antiautomorphism of the quaternion
algebra, denoted by the symbol $\gamma$ and defined as follows:%

\[
\gamma q=-a\overrightarrow{e_{1}}-b\overrightarrow{e_{2}}-c\overrightarrow
{e_{3}}+de_{4}\text{ for every }q=a\overrightarrow{e_{1}}+b\overrightarrow
{e_{2}}+c\overrightarrow{e_{3}}+de_{4}\in Q_{H}%
\]
which is called the \textquotedblleft Hamiltonian
conjugation\textquotedblright. This defines the symmetric, real valued,
non-degenerate, bilinear form $h$ are follows:%

\[
h(p,\text{ }q)=\frac{1}{2}\left[  \text{ }p\times \gamma q+q\times \gamma
p\right]  \text{ for }p,q\in Q_{H}.
\]
And then, the norm of any $q$ real quaternion denotes
\begin{equation}
\left \Vert q\right \Vert ^{2}=h(q,q)=q\times \gamma q. \label{2-3}%
\end{equation}

The concept of a spatial quaternion will be used of throughout our work. $q$
is called a spatial quaternion whenever $q+\gamma q=0.$ $\left[  2\right]  .$

The Serret-Frenet formulae for quaternionic curves in $E^{3}$ and $E^{4}$ are follows:

\begin{theorem}
\label{teo 2.1}The three-dimensional Euclidean space $E^{3}$ is identified
with the space of spatial quaternions $\left \{  p\in Q_{H}\text{ }\left \vert
\text{ }p+\gamma p=0\right.  \right \}  $ in an obvious manner. Let $I=\left[
0,1\right]  $ denotes the unit interval of the real line $%
\mathbb{R}
.$ Let
\[
\alpha:I\subset%
\mathbb{R}
\longrightarrow Q_{H}%
\]%
\[
s\longrightarrow \alpha(s)=\underset{i=1}{\overset{3}{\sum}}\alpha
_{i}(s)\overrightarrow{e_{i}},\text{ }1\leq i\leq3.
\]
be an arc-lenghted curve with nonzero curvatures $\left \{  k,\text{
}r\right \}  $ and $\left \{  t(s),\text{ }n(s),\text{ }b(s)\right \}  $ denotes
the Frenet frame of the curve $\alpha$. Then Frenet formulas are given by%
\begin{equation}
\left[
\begin{array}
[c]{c}%
t^{^{\prime}}\\
n^{^{\prime}}\\
b^{\prime}%
\end{array}
\right]  =\left[
\begin{array}
[c]{ccc}%
0 & k & 0\\
-k & 0 & r\\
0 & -r & 0
\end{array}
\right]  \left[
\begin{array}
[c]{c}%
t\\
n\\
b
\end{array}
\right]  \label{2-4}%
\end{equation}
where $k$ is the principal curvature, $r$ is torsion of $\alpha$ \cite{nag}$.$
\end{theorem}

\begin{theorem}
\label{teo 2.2}The four-dimensional Euclidean spaces $E^{4}$ is identified
with the space of quaternions. Let $I=\left[  0,1\right]  $ denotes the unit
interval of the real line $%
\mathbb{R}
.$ Let
\[
\alpha^{\left(  4\right)  }:I\subset%
\mathbb{R}
\longrightarrow Q_{H}%
\]%
\[
s\longrightarrow \alpha^{\left(  4\right)  }(s)=\underset{i=1}{\overset{4}%
{\sum}}\alpha_{i}(s)\overrightarrow{e_{i}},\text{ }1\leq i\leq4,\text{
}\overrightarrow{e_{4}}=1.
\]
be a smooth curve in $E^{4}$ with nonzero curvatures $\left \{  K,\text{
}k,\text{ }r-K\right \}  $ and $\left \{  T(s),\text{ }N(s),\text{ }%
B_{1}(s),\text{ }B_{2}(s)\right \}  $ denotes the Frenet frame of the curve
$\alpha$. Then Frenet formulas are given by%
\begin{equation}
\left[
\begin{array}
[c]{c}%
T^{^{\prime}}\\
N^{^{\prime}}\\
B_{1}^{\prime}\\
B_{2}^{\prime}%
\end{array}
\right]  =\left[
\begin{array}
[c]{cccc}%
0 & K & 0 & 0\\
-K & 0 & k & 0\\
0 & -k & 0 & (r-K)\\
0 & 0 & -(r-K) & 0
\end{array}
\right]  \left[
\begin{array}
[c]{c}%
T\\
N\\
B_{1}\\
B_{2}%
\end{array}
\right]  \label{2-5}%
\end{equation}
where $K$ is the principal curvature, $k$ is the torsion and $(r-K)$ is
bitorsion of $\alpha^{\left(  4\right)  }$ \cite{nag}$.$
\end{theorem}

\section{Characterizations of the Quaternionic Mannheim curve}

In this section, we define a quaternionic Mannheim curve\textbf{\ }and we give
some characterizations of them.in Euclidean $3$ and $4$ space.

\begin{definition}
\label{tan 3.1}Let $\alpha \left(  s\right)  $ and $\beta \left(  s^{\ast
}\right)  $ be two spatial quaternionic curves in $\mathbb{E}^{3}.$ $\left \{
t(s),\text{ }n(s),\text{ }b(s)\right \}  $ and $\left \{  t^{\ast}(s^{\ast
}),\text{ }n^{\ast}(s^{\ast}),\text{ }b^{\ast}(s^{\ast})\right \}  $ are Frenet
frames,respectively, on these curves. $\alpha \left(  s\right)  $ and
$\beta \left(  s^{\ast}\right)  $ are called spatial quaternionic Mannheim
curves if $n(s)$ and $b^{\ast}(s^{\ast})$ are linearly dependent.
\end{definition}

\begin{theorem}
\label{teo 3.1}Let $\alpha:I\subset \mathbb{R}\rightarrow$ $\mathbb{E}^{3}$ be
a spatial quaternionic Mannheim curve with the arc lenght parameter $s$ and
$\beta:I\subset \mathbb{R}\rightarrow \mathbb{E}^{3}$ be spatial quaternionic
Mannheim partner curve of $\alpha$ with the arc lenght parameter $s^{\ast}$.
Then
\[
d\left(  \alpha \left(  s\right)  ,\beta \left(  s^{\ast}\right)  \right)
=\text{constant, \  \ for all }s\in I
\]

\end{theorem}

\begin{proof}
From Definition $\left(  \text{\ref{tan 3.1}}\right)  $, we can write
\begin{equation}
\alpha \left(  s\right)  =\beta \left(  s^{\ast}\right)  +\lambda \left(
s^{\ast}\right)  b\left(  s^{\ast}\right)  \label{3-1}%
\end{equation}
Differentiating the \ Eq. \eqref{3-1} with respect to $s^{\ast}$ and by using
the Frenet equation$,$ we get
\[
\frac{d\alpha \left(  s\right)  }{ds}\frac{ds}{ds^{\ast}}=t^{\ast}(s^{\ast
})+\lambda^{\shortmid}(s^{\ast})b^{\ast}(s^{\ast})-\lambda^{\ast}(s^{\ast
})r^{\ast}(s^{\ast})n^{\ast}(s^{\ast})
\]
If we denote $\frac{d\alpha \left(  s\right)  }{ds}=t(s)$%
\[
t(s)=\frac{ds^{\ast}}{ds}\left[  t^{\ast}(s^{\ast})+\lambda^{\shortmid
}(s^{\ast})b^{\ast}(s^{\ast})-\lambda^{\ast}(s^{\ast})r^{\ast}(s^{\ast
})n^{\ast}(s^{\ast})\right]
\]
and
\[
h\left(  t(s),n\left(  s\right)  \right)  =\frac{ds^{\ast}}{ds}\left[
\begin{array}
[c]{c}%
h\left(  t^{\ast}(s^{\ast}),n\left(  s\right)  \right)  +\lambda^{\ast \prime
}\left(  s\right)  h\left(  b^{\ast}(s^{\ast}),n\left(  s\right)  \right) \\
-\lambda^{\ast}(s^{\ast})r^{\ast}(s^{\ast})h\left(  n^{\ast}(s^{\ast
}),n\left(  s\right)  \right)
\end{array}
\right]
\]
Since $\left \{  n\left(  s\right)  ,b^{\ast}(s^{\ast})\right \}  $ is a
linearly dependent set, we get
\[
\lambda^{\ast \prime}\left(  s\right)  =0
\]
that is, $\lambda^{\ast}$ is constant function on $I.$ This completes the proof.
\end{proof}

\begin{theorem}
\label{teo 3.2} Let $\left \{  \alpha,\beta \right \}  $ be a Mannheim curve
couple in $\mathbb{E}^{3}.$Then measure of the angle between the tangent
vector fields of spatial quaternionic curves $\alpha(s)$ and $\beta(s^{\ast})$
is constant.
\end{theorem}

\begin{proof}
Let $\alpha:I\subset \mathbb{R}\rightarrow$ $\mathbb{E}^{3}$ and $\beta
:I\subset \mathbb{R}\rightarrow$ $\mathbb{E}^{3}$ be spatial quaternionic
curves with arc-length $s$ and $s^{\ast}$ respectively. We show that
\begin{equation}
h\left(  t(s),t^{\ast}(s^{\ast})\right)  =\cos \theta=\text{constant}
\label{3-2}%
\end{equation}
Differentiating Eq. \eqref{3-2} with respect to $s,$ we get
\begin{align*}
\frac{d}{ds}h\left(  t(s),t^{\ast}(s^{\ast})\right)   &  =h\left(
\frac{dt(s)}{ds},t^{\ast}(s^{\ast})\right)  +h\left(  t(s),\frac{dt^{\ast
}(s^{\ast})}{ds^{\ast}}\frac{ds^{\ast}}{ds}\right) \\
&  =h\left(  k(s)n(s),t^{\ast}(s^{\ast})\right)  +h\left(  t(s),k^{\ast
}(s^{\ast})n^{\ast}(s^{\ast})\frac{ds^{\ast}}{ds}\right) \\
&  =0
\end{align*}
Thus,%
\[
h\left(  t(s),t^{\ast}(s^{\ast})\right)  =\text{constant}%
\]

\end{proof}

\begin{theorem}
\label{teo 3.2.1}Let $\alpha:I\subset \mathbb{R}\rightarrow$ $\mathbb{E}^{3}$
be spatial quaternionic curves with the arc-length parameter $s$. Then
$\alpha$ is spatial quaternionic Mannheim curve if and only if
\[
k(s)=\lambda \left(  k^{2}(s)+r^{2}(s)\right)
\]
where $\lambda_{1},$ $\lambda_{2}$ are constants.
\end{theorem}

\begin{proof}
If $\alpha$ is spatial quaternionic Mannheim curve, we can write
\[
\beta \left(  s\right)  =\alpha \left(  s\right)  +\lambda(s)n\left(  s\right)
\]
Differentiating the above equality and by using the Frenet equations, we get
\[
\frac{d\beta(s)}{ds}=\left[  \left(  1-\lambda(s)k\left(  s\right)  \right)
t(s)+\lambda^{\shortmid}(s)n(s)+\lambda(s)r\left(  s\right)  b\left(
s\right)  \right]
\]
as $\left \{  n\left(  s\right)  ,b^{\ast}(s^{\ast})\right \}  $ is a linearly
dependent set, we get
\[
\lambda^{\shortmid}(s)=0.
\]
This means that $\lambda$ is constant. Thus we have%
\[
\frac{d\beta(s)}{ds}=\left(  1-\lambda k\left(  s\right)  \right)
t(s)+\lambda r\left(  s\right)  b\left(  s\right)  .
\]
On the other hand, we have%
\[
t^{\ast}=\frac{d\beta}{ds}\frac{ds}{ds^{\ast}}=\left[  \left(  1-\lambda
k\left(  s\right)  \right)  t(s)+\lambda r\left(  s\right)  b\left(  s\right)
\right]  \frac{ds}{ds^{\ast}}.
\]
By taking the derivative of this equation with respect to $s^{\ast}$ and
appliying the Frenet formulas we obtain%
\begin{align*}
\frac{dt^{\ast}}{ds}\frac{ds}{ds^{\ast}}  &  =\left[  -\lambda k^{\shortmid
}(s)t(s)+\left(  k(s)-\lambda k^{2}(s)-\lambda r^{2}(s)\right)  n(s)+\lambda
r^{\shortmid}(s)b(s)\right]  \left(  \frac{ds}{ds^{\ast}}\right)  ^{2}\\
&  +\left[  \left(  1-\lambda k\left(  s\right)  \right)  t(s)+\lambda
r\left(  s\right)  b\left(  s\right)  \right]  \frac{d^{2}s}{ds^{\ast^{2}}}%
\end{align*}
From this equation we get%
\[
k(s)=\lambda \left(  k^{2}(s)+r^{2}(s)\right)  .
\]
Conversely, if $k(s)=\lambda \left(  k^{2}(s)+r^{2}(s)\right)  ,$ then we can
easily see that $\alpha$ is a Mannheim curve.
\end{proof}

\begin{theorem}
\label{teo 3.3}Let $\alpha:I\subset \mathbb{R}\rightarrow$ $\mathbb{E}^{3}$ be
spatial quaternionic Mannheim curve with arc-length parameter $s$.$\ $Then
$\beta$ is the spatial quaternionic Mannheim partner curve of $\alpha$ if and
only if the curvature functions $k^{\ast}\left(  s^{\ast}\right)  $ and
$r^{\ast}(s^{\ast})$ of $\beta$ satisfy the following equation%
\[
\frac{dr^{\ast}}{ds^{\ast}}=\frac{k^{\ast}}{\mu}\left(  1+\mu^{2}r^{\ast^{2}%
}\right)  ,
\]
where $\mu$ is constant.
\end{theorem}

\begin{proof}
Let $\alpha:I\subset \mathbb{R}\rightarrow$ $\mathbb{E}^{3}$ be spatial
quaternionic Mannheim curve. Then, we can write
\begin{equation}
\alpha \left(  s^{\ast}\right)  =\beta \left(  s^{\ast}\right)  +\mu \left(
s^{\ast}\right)  b^{\ast}\left(  s^{\ast}\right)  \label{3-3}%
\end{equation}
for some function $\mu \left(  s^{\ast}\right)  $. By taking the derivative of
this equation with respect to $s^{\ast}$\ and using the Frenet equations we
obtain%
\[
t\frac{ds}{ds^{\ast}}=t^{\ast}\left(  s^{\ast}\right)  +\mu^{\shortmid}\left(
s^{\ast}\right)  b^{\ast}\left(  s^{\ast}\right)  -\mu \left(  s^{\ast}\right)
r^{\ast}\left(  s^{\ast}\right)  n^{\ast}\left(  s^{\ast}\right)  .
\]
And then, we know that $\left \{  n\left(  s\right)  ,b^{\ast}\left(  s^{\ast
}\right)  \right \}  $ is a linearly dependent set, so we have
\[
\mu^{\shortmid}\left(  s^{\ast}\right)  =0.
\]
This means that $\mu \left(  s^{\ast}\right)  $ is a constant function. Thus we
have%
\begin{equation}
t\frac{ds}{ds^{\ast}}=t^{\ast}\left(  s^{\ast}\right)  -\mu r^{\ast}\left(
s^{\ast}\right)  n^{\ast}\left(  s^{\ast}\right)  . \label{3-4}%
\end{equation}
On the other hand, we have%
\begin{equation}
t=t^{\ast}\cos \theta+n^{\ast}\sin \theta \label{3-5}%
\end{equation}
where $\theta$ is the angle between $t$ and $t^{\ast}$ at the corresponding
points of $\alpha$ and $\beta.$ By taking the derivative of this equation with
respect to $\overline{s}$ and using the Frenet equations we obtain%
\[
kn\frac{ds}{ds^{\ast}}=-\left(  k^{\ast}+\frac{d\theta}{ds^{\ast}}\right)
\sin \theta t^{\ast}+\left(  k^{\ast}+\frac{d\theta}{ds^{\ast}}\right)
\cos \theta n^{\ast}+r^{\ast}\sin \theta b^{\ast}.
\]
From this equation and the fact that the $\left \{  n\left(  s\right)
,b^{\ast}\left(  s^{\ast}\right)  \right \}  $ is a linearly dependent set, we
get%
\[
\left \{
\begin{array}
[c]{c}%
\left(  k^{\ast}+\frac{d\theta}{ds^{\ast}}\right)  \sin \theta=0\\
\left(  k^{\ast}+\frac{d\theta}{ds^{\ast}}\right)  \cos \theta=0.
\end{array}
\right.
\]
For this reason we have%
\begin{equation}
\frac{d\theta}{ds^{\ast}}=-k^{\ast}. \label{3-6}%
\end{equation}
From the Eq. \eqref{3-4} and Eq. \eqref{3-5} and notice that $t^{\ast}$ is
orthogonal to $b^{\ast}$, we find that%
\[
\frac{ds}{ds^{\ast}}=\frac{1}{\cos \theta}=-\frac{\mu r^{\ast}}{\sin \theta}.
\]
Then we have%
\[
\mu r^{\ast}=-\tan \theta.
\]
By taking the derivative of this equation and applying Eq. $(3.6)$, we get%
\[
\mu \frac{dr^{\ast}}{ds^{\ast}}=k^{\ast}\left(  1+\mu^{2}r^{\ast^{2}}\right)
\]
that is%
\[
\frac{dr^{\ast}}{ds^{\ast}}=\frac{k^{\ast}}{\mu}\left(  1+\mu^{2}r^{\ast^{2}%
}\right)  .
\]
Conversely, if the curvature $k^{\ast}$ and torsion $r^{\ast}$ of the curve
$\beta$ satisfy%
\[
\frac{dr^{\ast}}{ds^{\ast}}=\frac{k^{\ast}}{\mu}\left(  1+\mu^{2}r^{\ast^{2}%
}\right)
\]
for constant $\mu,$ then we define a curve a curve by%
\begin{equation}
\alpha \left(  s^{\ast}\right)  =\beta \left(  s^{\ast}\right)  +\mu b^{\ast
}\left(  s^{\ast}\right)  \label{3-7}%
\end{equation}
and we will show that $\alpha$ is a\ spatial quaternionic Mannheim curve and
$\beta$ is the spatial quaternionic partner curve of $\alpha$. By taking the
derivative of Eq. $(3.7)$ with respect to $\overline{s}$ twice, we get%
\begin{equation}
t\frac{ds}{ds^{\ast}}=t^{\ast}-\mu r^{\ast}n^{\ast}, \label{3-8}%
\end{equation}%
\begin{equation}
kn\left(  \frac{ds}{ds^{\ast}}\right)  ^{2}+t\frac{d^{2}s}{ds^{\ast^{2}}}=\mu
k^{\ast}r^{\ast}t^{\ast}+\left(  k^{\ast}-\mu \frac{dr^{\ast}}{ds^{\ast}%
}\right)  n^{\ast}-\mu r^{\ast^{2}}b^{\ast} \label{3-9}%
\end{equation}
respectively. Taking the cross product of Eq. \eqref{3-8} with Eq. \eqref{3-9}
and noticing that%
\[
k^{\ast}-\mu \frac{dr^{\ast}}{ds^{\ast}}+\mu^{2}k^{\ast}r^{\ast^{2}}=0,
\]
we have%
\begin{equation}
kb\left(  \frac{ds}{ds^{\ast}}\right)  ^{3}=\mu^{2}r^{\ast^{3}}t^{\ast}+\mu
r^{\ast^{2}}n^{\ast}. \label{3-10}%
\end{equation}
By taking the cross product of Eq. \eqref{3-8} with Eq. \eqref{3-10}, we get%
\[
kn\left(  \frac{ds}{ds^{\ast}}\right)  ^{4}=-\mu r^{\ast^{2}}\left(  1+\mu
^{2}r^{\ast^{2}}\right)  b^{\ast}.
\]
This means that the principal normal vector field of the spatial quaternionic
curve $\alpha$ and binormal vector field of the spatial quaternionic curve
$\beta$ are linearly dependent set. And so $\alpha$ is a spatial quaternionic
Mannheim curve and $\beta$ is spatial quaternionic Mannheim partner curve of
$\alpha$.
\end{proof}

\begin{definition}
\label{tan 3.2}A quaternionic curve $\alpha^{\left(  4\right)  }%
:I\subset \mathbb{R}\rightarrow$ $\mathbb{E}^{4}$ is a quaternionic Mannheim
curve if there exists a quaternionic curve $\beta^{\left(  4\right)
}:\overline{I}\subset \mathbb{R}\rightarrow$ $\mathbb{E}^{4}$ such that the
second Frenet vector at each point of $\alpha^{\left(  4\right)  }$ is
included the plane generated by the third Frenet vector and the fourth Frenet
vector of $\beta^{\left(  4\right)  }$ at corresponding point under $\varphi,$
where $\varphi$ is a bijection from $\alpha^{\left(  4\right)  }$ to
$\beta^{\left(  4\right)  }$. The curve $\beta^{\left(  4\right)  }$ is called
the quaternionic Mannheim partner curve of $\alpha^{\left(  4\right)  }.$
\end{definition}

\begin{theorem}
\label{teo 3.4}Let $\alpha^{\left(  4\right)  }:I\subset \mathbb{R}\rightarrow$
$\mathbb{E}^{4}$ and $\beta^{\left(  4\right)  }:\overline{I}\subset
\mathbb{R}\rightarrow$ $\mathbb{E}^{4}$ be quaternionic Mannheim curve couple
with arc-length $s$ and $\overline{s},$ respectively$.$ Then
\begin{equation}
d\left(  \alpha^{\left(  4\right)  }\left(  s\right)  ,\beta^{\left(
4\right)  }\left(  \overline{s}\right)  \right)  =\lambda(s)=\text{constant,
\ for all }s\in I \label{3-11}%
\end{equation}

\end{theorem}

\begin{proof}
From the Definition $\left(  \text{\ref{tan 3.2}}\right)  $, quaternionic
Mannheim partner curve $\beta^{\left(  4\right)  }$ of \ $\alpha^{\left(
4\right)  }$ is given by the following equation%
\[
\beta^{\left(  4\right)  }\left(  s\right)  =\alpha^{\left(  4\right)
}(s)+\lambda(s)N\left(  s\right)  .
\]
where $\lambda \left(  s\right)  $ is a smooth function. A smooth function
$\psi:s\in I\rightarrow \psi \left(  s\right)  =\overline{s}\in \overline{I}$ is
defined by%
\[
\psi \left(  s\right)  =\overset{s}{\underset{0}{\int}}\left \Vert \frac
{d\alpha^{\left(  4\right)  }\left(  s\right)  }{ds}\right \Vert ds=\overline
{s}.
\]
The bijection $\varphi$:$\alpha^{\left(  4\right)  }\rightarrow \beta^{\left(
4\right)  }$ is defined by $\varphi \left(  \alpha^{\left(  4\right)
}(s)\right)  =\beta^{\left(  4\right)  }(\psi \left(  s\right)  ).$ Since the
second Frenet vector at each point of $\alpha^{\left(  4\right)  }$ is
included the plane generated by the third Frenet vector and the fourth Frenet
vector of $\beta^{\left(  4\right)  }$ at corresponding point under $\varphi,$
for each $s\in I,$ the Frenet vector $N(s)$ is given by the linear combination
of Frenet vectors $\overline{B}_{1}\left(  \psi \left(  s\right)  \right)  $
and $\overline{B}_{2}\left(  \psi \left(  s\right)  \right)  $ of
$\beta^{\left(  4\right)  },$ that is, we can write%
\[
N(s)=g(s)\overline{B}_{1}\left(  \psi \left(  s\right)  \right)  +h(s)\overline
{B}_{2}\left(  \psi \left(  s\right)  \right)  ,
\]
where $g(s)$ and $h(s)$ are smooth functions on $I.$ So we can write%
\begin{equation}
\beta^{\left(  4\right)  }\left(  \psi(s)\right)  =\alpha^{\left(  4\right)
}\left(  s\right)  +\lambda(s)N\left(  s\right)  . \label{3-12}%
\end{equation}
Differentiating Eq. \eqref{3-12} with respect to $s$ and by using the Frenet
equations, we get
\[
\overline{T}\left(  \psi \left(  s\right)  \right)  \psi^{\shortmid}(s)=\left[
\left(  1-\lambda K\left(  s\right)  \right)  T(s)+\lambda^{\prime}\left(
s\right)  N\left(  s\right)  +\lambda \left(  s\right)  k\left(  s\right)
B_{1}\left(  s\right)  \right]  .
\]
By the fact that:%
\[
h\left(  \overline{T}\left(  \varphi \left(  s\right)  \right)  ,g(s)\overline
{B}_{1}\left(  \psi \left(  s\right)  \right)  +h(s)\overline{B}_{2}\left(
\psi \left(  s\right)  \right)  \right)  =0,
\]
we have%
\[
\lambda^{\prime}\left(  s\right)  =0
\]
that is, $\lambda \left(  s\right)  $ is constant function on $I.$ This
completes the proof.
\end{proof}

\begin{theorem}
\label{teo 3.5}If the quaternionic curve $\alpha^{\left(  4\right)  }%
:I\subset \mathbb{R}\rightarrow$ $\mathbb{E}^{4}$ is a quaternionic Mannheim
curve, then the first and second curvature functions of $\alpha^{\left(
4\right)  }$ satisfy the equality:%
\[
K(s)=\lambda \left \{  K^{2}(s)+k^{2}(s)\right \}
\]
where $\lambda$ is constant.
\end{theorem}

\begin{proof}
Let $\beta^{\left(  4\right)  }$ be a quaternionic Mannheim partner curve of
$\alpha^{\left(  4\right)  }.$ Then we can write%
\[
\beta^{\left(  4\right)  }\left(  \psi \left(  s\right)  \right)
=\alpha^{\left(  4\right)  }\left(  s\right)  +\lambda N\left(  s\right)
\]
Differentiating, we get
\[
\overline{T}\left(  \psi \left(  s\right)  \right)  \psi^{\shortmid}(s)=\left[
\left(  1-\lambda K\left(  s\right)  \right)  T(s)+\lambda k\left(  s\right)
B_{1}\left(  s\right)  \right]  ,
\]
taht is,%
\[
\overline{T}\left(  \psi \left(  s\right)  \right)  =\frac{1-\lambda K\left(
s\right)  }{\psi^{\shortmid}(s)}T(s)+\frac{\lambda k\left(  s\right)  }%
{\psi^{\shortmid}(s)}B_{1}\left(  s\right)
\]
where $\psi^{\shortmid}(s)=\sqrt{\left(  1-\lambda K(s)\right)  ^{2}+\left(
\lambda k(s)\right)  ^{2}}$ for $s\in I.$ By differentiation of both sides of
the above equality with respect to $s$, we have%
\begin{align*}
\psi^{\shortmid}(s)\overline{K}\left(  \overline{s}\right)  \overline
{N}\left(  \overline{s}\right)   &  =\left(  \frac{1-\lambda K\left(
s\right)  }{\psi^{\shortmid}(s)}\right)  ^{\shortmid}T(s)\\
&  +\left(  \frac{\left(  1-\lambda K\left(  s\right)  \right)  K(s)-\lambda
k\left(  s\right)  ^{2}}{\psi^{\shortmid}(s)}\right)  N(s)\\
&  +\left(  \frac{\lambda k\left(  s\right)  }{\psi^{\shortmid}(s)}\right)
^{\shortmid}B_{1}\left(  s\right)  -\frac{\lambda k\left(  s\right)  \left(
r(s)-K(s)\right)  }{\psi^{\shortmid}(s)}B_{2}\left(  s\right)  .
\end{align*}
By the fact:%
\[
h\left(  \overline{N}\left(  \varphi \left(  s\right)  \right)  ,g(s)\overline
{B}_{1}\left(  \psi \left(  s\right)  \right)  +h(s)\overline{B}_{2}\left(
\psi \left(  s\right)  \right)  \right)  =0,
\]
we have that coefficient of $N$ in the above equation is zero, that is,%
\[
\left(  1-\lambda K\left(  s\right)  \right)  K(s)-\lambda k\left(  s\right)
^{2}=0.
\]
Thus, we have%
\[
K(s)=\lambda \left \{  K^{2}(s)+k^{2}(s)\right \}
\]
for $s\in I$. This completes the proof.
\end{proof}

\begin{theorem}
\label{teo 3.6}Let $\alpha^{\left(  4\right)  }:I\subset \mathbb{R}\rightarrow$
$\mathbb{E}^{4}$ be quaternionic curven with arc-length $s$ whose curvature
functions $K(s)$ and $k(s)$ are non-constant functions and satisfy the
equality: $K(s)=\lambda \left \{  K^{2}(s)+k^{2}(s)\right \}  ,$ where $\lambda$
is constant. If the quatern\i onic curve $\beta^{\left(  4\right)  }$ is given
by $\beta^{\left(  4\right)  }\left(  \overline{s}\right)  =\alpha^{\left(
4\right)  }\left(  s\right)  +\lambda N\left(  s\right)  ,$ then
$\alpha^{\left(  4\right)  }$ is a quaternionic Mannheim curve and
$\beta^{\left(  4\right)  }$ is the quaternionic Mannheim partner curve of
$\alpha^{\left(  4\right)  }$.
\end{theorem}

\begin{proof}
Let $\overline{s}$ be the arc-length of the quaternionic curve $\beta^{\left(
4\right)  }$. That is, $\overline{s}$ is defined by%
\[
\overline{s}=\overset{s}{\underset{0}{\int}}\left \Vert \frac{d\alpha^{\left(
4\right)  }\left(  s\right)  }{ds}\right \Vert ds
\]
for $s\in I$. We can write a smooth function $\psi:s\in I\rightarrow
\psi \left(  s\right)  =\overline{s}\in \overline{I}$. By the assumption of the
curvature functions $K(s)$ and $k(s)$, we have%
\begin{align*}
\psi^{\shortmid}(s)  &  =\sqrt{\left(  1-\lambda K(s)\right)  ^{2}+\left(
\lambda k(s)\right)  ^{2}},\\
\psi^{\shortmid}(s)  &  =\sqrt{1-\lambda K(s)}%
\end{align*}
for $s\in I$. Then we can easily write%
\begin{align*}
\beta^{\left(  4\right)  }\left(  \overline{s}\right)   &  =\beta^{\left(
4\right)  }\left(  \psi \left(  s\right)  \right) \\
&  =\alpha^{\left(  4\right)  }\left(  s\right)  +\lambda N\left(  s\right)
\end{align*}
for the quaternionic curve $\beta^{\left(  4\right)  }$. If we differentiate
both sides of the above equality with respect to $s,$ we get%
\[
\psi^{\shortmid}(s)\overline{T}\left(  \psi(s)\right)  =T(s)+\lambda \left \{
-K(s)T(s)+k(s)B_{1}(s)\right \}  .
\]
And so we have,%
\begin{equation}
\overline{T}\left(  \psi(s)\right)  =\sqrt{1-\lambda K(s)}T(s)+\frac{\lambda
k(s)}{\sqrt{1-\lambda K(s)}}B_{1}(s). \label{3-13}%
\end{equation}
Differentiating the above equality with respect to $s$ and by using the Frenet
equations, we get%
\begin{align*}
\psi^{\shortmid}(s)\overline{K}\left(  \psi(s)\right)  \overline{N}\left(
\psi(s)\right)   &  =\left(  \sqrt{1-\lambda K(s)}\right)  ^{\shortmid
}T(s)+\left(  \frac{K(s)\left(  1-\lambda K(s)\right)  -\lambda k^{2}%
(s)}{\sqrt{1-\lambda K(s)}}\right)  N(s)\\
&  +\left(  \frac{\lambda k(s)}{\sqrt{1-\lambda K(s)}}\right)  ^{\shortmid
}B_{1}(s)+\frac{\lambda k(s)\left(  r(s)-K(s)\right)  }{\sqrt{1-\lambda K(s)}%
}B_{2}(s)
\end{align*}
From our assumption, it holds%
\[
\frac{K(s)\left(  1-\lambda K(s)\right)  -\lambda k^{2}(s)}{\sqrt{1-\lambda
K(s)}}=0.
\]
We find the coefficient of $N(s)$ in the above equality vanishes. Thus the
vector $\overline{N}\left(  \psi(s)\right)  $ is given by linear combination
of $T(s),$ $B_{1}(s)$ and $B_{2}(s)$ for each $s\in I.$ And the vector
$\overline{T}\left(  \psi(s)\right)  $ is given by linear combination of
$T(s)$ and $B_{1}(s)$ for each $s\in I$ in the Eq. \eqref{3-13}. As the curve
$\beta^{\left(  4\right)  }$ is quaternionic curve in $\mathbb{E}^{4}$, the
vector $N(s)$ is given by linear combination of $\overline{B}_{1}\left(
\overline{s}\right)  $ and $\overline{B}_{2}\left(  \overline{s}\right)  $.
For this reason, the second Frenet curve at each point of $\alpha^{(4)}$ is
included in the plane generated the third Frenet vector and the fourth Frenet
vector of $\beta^{\left(  4\right)  }$ at corresponding point under $\varphi$.
Here the bijection $\varphi:\alpha^{\left(  4\right)  }\rightarrow
\beta^{\left(  4\right)  }$ is defined by $\varphi \left(  \alpha^{\left(
4\right)  }(s)\right)  =\beta^{\left(  4\right)  }(\psi \left(  s\right)  )$.
This completes the proof.
\end{proof}

\end{document}